\documentclass[11pt]{amsart}

\scrollmode

\usepackage{times}
\usepackage{amsmath,graphicx}
\usepackage{latexsym}
\usepackage{amscd,amsthm,amssymb}
\usepackage[all]{xy}

\newtheorem{thm}{Theorem}[section]
\newtheorem{prop}[thm]{Proposition}
\newtheorem{lem}[thm]{Lemma}

\newtheorem{conj}{Conjecture}

\theoremstyle{definition}

\newtheorem{defn}[thm]{Definition}

\theoremstyle{remark}

\numberwithin{equation}{section}

\title{The closure of the symplectic cone of elliptic surfaces}
\author{M.~J.~D.~Hamilton}
\address{      Institute for geometry and topology\\
               University of Stuttgart\\
               Pfaffenwaldring 57\\
               70569 Stuttgart\\
               Germany}
\email{mark.hamilton@math.lmu.de}
\date{\today}
\subjclass[2010]{Primary 14J27, 57R17; Secondary 57R50}
\keywords{4-manifold, symplectic cone, elliptic surface, diffeomorphism group}

\begin{document}

\begin{abstract}The symplectic cone of a closed oriented 4-manifold is the set of cohomology classes represented by symplectic forms. A well-known conjecture describes this cone for every minimal K\"ahler surface. We consider the case of the elliptic surfaces $E(n)$ and focus on a slightly weaker conjecture for the closure of the symplectic cone. We prove this conjecture in the case of the spin surfaces $E(2m)$ using inflation and the action of self-diffeomorphisms of the elliptic surface. An additional obstruction appears in the non-spin case.
\end{abstract}

\maketitle

\section{Introduction}

Let $M$ be a closed oriented 4-manifold. We are interested in the set $\mathcal{C}_M$ of real cohomology classes represented by symplectic forms on $M$, called the {\em symplectic cone} of $M$. It is indeed a cone because a non-zero multiple of any symplectic form is again symplectic. We only consider symplectic forms $\omega$ compatible with the orientation, so that $\omega\wedge\omega$ is everywhere positive. It follows that the symplectic cone is a subset of the {\em positive cone} $\mathcal{P}$, given by the set  of elements in $H^2(M;\mathbb{R})$ which have positive square. In fact, according to the proof of Observation 4.3 in \cite{G}, the symplectic cone is always an open subset of the positive cone. If the 4-manifold $M$ does not admit a symplectic form then the set $\mathcal{C}_M$ is empty. It is also useful to denote by $\mathcal{P}^A$ for a non-zero cohomology class $A\in H^2(M;\mathbb{R})$ the set of elements in $\mathcal{P}$ which have positive cup product with $A$. Clearly, $\mathcal{P}^A\cup \mathcal{P}^{-A}$ is a cone. In addition, we set $\mathcal{P}^0=\mathcal{P}$.

The symplectic cone has been determined in the following cases:
\begin{enumerate}
\item $S^2$-bundles over surfaces \cite{McD}.
\item $T^2$-bundles over $T^2$ \cite{Ge}.
\item All 4-manifolds with a fixed point free circle action \cite{Bo, FV1, FV2}.
\item All symplectic 4-manifolds with $b_2^+=1$ \cite{LiLiu}.
\item The $K3$ surface \cite{Li}.
\item Fibre sums along tori of $T^2\times \Sigma_g$ and minimal elliptic K\"ahler surfaces with $b_2^+=1$, for example Enriques or Dolgachev surfaces \cite{DL}.
\end{enumerate}

The simply-connected 4-manifolds among these cases either have $b_2^+=1$ or are diffeomorphic to the $K3$ surface, because the 4-manifolds in (c) have zero Euler characteristic.

From now on we denote by $M$ a simply-connected elliptic surface $E(n)$ without multiple fibres and with the complex orientation. Since by the results mentioned above the symplectic cone is known for $E(1)=\mathbb{CP}^2\#9{\overline{\mathbb{CP}}{}^{2}}$  and the $K3$ surface $E(2)$ we assume that $n\geq 3$. Let $F$ denote the class of the fibre in an elliptic fibration on $M$. We set 
\begin{equation*}
c_1(M)=-(n-2)PD(F),
\end{equation*}
where $PD$ denotes the Poincar\'e dual of the homology class. Note that symplectic forms have well-defined Chern classes, defined by considering any compatible almost complex structure. If $\omega$ is a symplectic form with first Chern class $c_1(M,\omega)$, then $-\omega$ is a symplectic form with first Chern class $-c_1(M,\omega)$. It is known from Seiberg-Witten theory \cite{GS} that every symplectic form on $E(n)$ has up to sign first Chern class equal to $c_1(M)$. It is also known from the theorems of Taubes \cite{T} that for every symplectic structure $\omega$ the Poincar\'e dual of the class $-c_1(M,\omega)$ is represented by an embedded symplectic surface. This implies that $\omega\cdot c_1(M)$ is non-zero, hence the symplectic cone satisfies
\begin{equation*}
\mathcal{C}_M\subset \left(\mathcal{P}^{c_1(M)}\cup \mathcal{P}^{-c_1(M)}\right).
\end{equation*}
A well-known conjecture due to Tian-Jun Li \cite{Li} says that the following holds:
\begin{conj}[Strong conjecture]\label{strong conj}
We have
\begin{equation*}
\left(\mathcal{P}^{c_1(M)}\cup \mathcal{P}^{-c_1(M)}\right)\subset \mathcal{C}_M.
\end{equation*}
Hence every class of positive square whose cup product with the first Chern class of $M$ is non-zero is represented by a symplectic form.
\end{conj}
The conjecture should even hold for any closed 4-manifold underlying a minimal K\"ahler surface, but we only consider the case of elliptic surfaces. There is also a slightly weaker form of the conjecture. We denote by $\overline{\mathcal{C}}_M$ the closure of the symplectic cone in the vector space $H^2(M;\mathbb{R})$.
\begin{conj}[Weak conjecture]\label{main conj}
We have
\begin{equation*}
\left(\mathcal{P}^{c_1(M)}\cup \mathcal{P}^{-c_1(M)}\right)\subset \overline{\mathcal{C}}_M.
\end{equation*}
Hence every class of positive square whose cup product with the first Chern class of $M$ is non-zero is the limit of a sequence of symplectic classes. Equivalently, the symplectic cone $\mathcal{C}_M$ is dense in $\left(\mathcal{P}^{c_1(M)}\cup \mathcal{P}^{-c_1(M)}\right)$.
\end{conj}
In the following we only consider the weak conjecture. 
To state the theorem we want to prove, consider the following definition:
\begin{defn}
We define $\mathcal{P}^{>}\subset\mathcal{P}$ to be the subcone of elements $\omega$ with 
\begin{equation*}
\omega^2>(\omega\cdot PD(F))^2.
\end{equation*}
For a non-zero class $A\in H^2(M;\mathbb{R})$ we set 
\begin{equation*}
\mathcal{P}^{A>}=\mathcal{P}^{>}\cap \mathcal{P}^A.
\end{equation*}
In particular, this applies to $A=PD(F)$ and $A=\pm c_1(M)$. Note that $\mathcal{P}^{A>}\cup \mathcal{P}^{-A>}$ is a subcone of $\mathcal{P}^{A}\cup\mathcal{P}^{-A}$.
\end{defn}
Then we have:
\begin{thm}\label{main thm}
Let $m\geq 2$ be an integer. If $M$ is the spin surface $E(2m)$ then
\begin{equation*}
\left(\mathcal{P}^{c_1(M)}\cup \mathcal{P}^{-c_1(M)}\right)\subset \overline{\mathcal{C}}_M.
\end{equation*}
If $M$ is the non-spin surface $E(2m-1)$ then
\begin{equation*}
\left(\mathcal{P}^{c_1(M)>}\cup \mathcal{P}^{-c_1(M)>}\right)\subset \overline{\mathcal{C}}_M.
\end{equation*}
\end{thm}
This proves Conjecture \ref{main conj} in the case of the spin elliptic surfaces $E(2m)$. At the moment we do not know how to prove the full Conjecture \ref{main conj} in the non-spin case. One can view these results as evidence that the strong Conjecture \ref{strong conj} is indeed true. The sequences of symplectic forms in the theorem are all obtained from a single symplectic form by inflation along certain symplectic surfaces and the action of the orientation preserving self-diffeomorphisms of the elliptic surface $M$.

\subsection*{Acknowledgements} I would like to thank Tian-Jun Li for very helpful conversations.

\section{Some notation}

We follow the notation from \cite{H}. In particular, all self-diffeomorphisms of $M$ are orientation preserving. We often denote a symplectic form and its class by the same symbol. Note that considering minus a given symplectic form we see that to prove the weak conjecture it suffices to prove that
\begin{equation*}
\mathcal{P}^{PD(F)}\subset \overline{\mathcal{C}}_M.
\end{equation*}
We want to prove the following theorem, which is equivalent to Theorem \ref{main thm}:
\begin{thm}\label{main thm F}
Let $m\geq 2$ be an integer. If $M$ is the spin surface $E(2m)$ then
\begin{equation*}
\mathcal{P}^{PD(F)}\subset \overline{\mathcal{C}}_M.
\end{equation*}
If $M$ is the non-spin surface $E(2m-1)$ then
\begin{equation*}
\mathcal{P}^{PD(F)>}\subset \overline{\mathcal{C}}_M.
\end{equation*}
\end{thm}
We will first prove a special case of Theorem \ref{main thm F} since this is easier and uses the same method as in the general case. We need some notation. Consider the manifold $M=E(n)$ with $n\geq 3$ and define an integer $m$ by $n=2m$ if $n$ is even and $n=2m-1$ if $n$ is odd.
\begin{defn}
Let $W$ be the embedded surface obtained by smoothing the intersections of a section $V$ of the elliptic surface $M$ of square $-n$ and $m$ parallel copies of the fibre $F$. Let $R$ denote a rim torus of square zero and $S$ a dual vanishing sphere of square $-2$ in $M$. Both $F,W$ and $R,S$ intersect in a single transverse positive point. Otherwise the surfaces are disjoint.
\end{defn}
The vanishing sphere $S$ is obtained by sewing together in the fibre sum 
\begin{equation*}
E(n)=E(1)\#_{F=F}E(n-1)
\end{equation*}
two vanishing disks coming from singular fibres with the same vanishing cycles. The surface $W$ has self-intersection zero if $n$ is even and one if $n$ is odd. The surfaces $F$ and $W$ span a copy of the standard hyperbolic form $H$ in the intersection form if $n$ is even and a copy of $H'$ if $n$ is odd, where 
\begin{equation*}
H'=\left(\begin{array}{cc} 0 & 1 \\ 1 & 1 \end{array}\right).
\end{equation*}
We will denote both intersection forms by $H(n)$. Since this form is unimodular the total intersection form over the integers looks like 
\begin{equation*}
Q_M=H(n)\oplus H(n)^\perp.
\end{equation*}
We can decompose any class $\omega \in H^2(M;\mathbb{R})$ according to this splitting as 
\begin{equation*}
\omega=PD(\alpha F+\beta W) +\omega'
\end{equation*}
with $\omega'\in H(n)^\perp$ (here we mean the real subspace spanned by this lattice). Note that 
\begin{equation*}
\beta=\omega\cdot PD(F).
\end{equation*}

\begin{defn}
We define 
\begin{equation*}
\mathcal{P}^{PD(F)+}\subset\mathcal{P}^{PD(F)}
\end{equation*}
to be the subset of elements of the form $\omega=PD(\alpha F+\beta W) +\omega'$ where $\alpha,\beta$ and $\omega'^2$ are positive. We call the classes in this subset {\em positive}.
\end{defn}
\begin{thm}\label{thm PD(F)+}
We have
\begin{equation*}
\mathcal{P}^{PD(F)+}\subset \overline{\mathcal{C}}_M.
\end{equation*}
\end{thm}
This theorem describes the first subset of $\mathcal{P}^{PD(F)}$ that we can represent by the limits of symplectic forms. We will extend it later and prove Theorem \ref{main thm F}.

\section{Symplectic forms and diffeomorphisms}
The inflation procedure, introduced by Lalonde and McDuff \cite{LaMc,McD}, shows that if $\Sigma$ is a closed connected symplectic surface of non-negative square in a closed symplectic 4-manifold $(Y,\omega)$, then the class $[\omega]+tPD(\Sigma)$ is represented by a symplectic form for all $t\geq 0$. We need the following generalized inflation lemma:
\begin{lem}\label{inflation}
Let $(Y,\omega)$ be a closed symplectic 4-manifold and $\Sigma_1,\Sigma_2\subset Y$ closed connected symplectic surfaces of non-negative square which intersect transversely in a single positive point. Then for all real numbers $r_1,r_2\geq 0$ the class
\begin{equation*}
[\omega]+r_1PD(\Sigma_1)+r_2PD(\Sigma_2)
\end{equation*}
is represented by a symplectic form.
\end{lem}
\begin{proof} By the symplectic neighbourhood theorem $\Sigma_1$ has a tubular neighbourhood $\nu \Sigma_1$ with symplectic fibres. According to Lemma 2.3 in \cite{G} we can assume that $\Sigma_2$ intersects $\nu \Sigma_1$ in one of the disk fibres. If we first do inflation along $\Sigma_1$ as in \cite[Lemma 3.7]{McD} then the symplectic form changes only in the tubular neighbourhood $\nu \Sigma_1$ and the fibres stay symplectic. Hence $\Sigma_2$ remains symplectic and we can then do inflation along $\Sigma_2$. Compare with \cite[Lemma 2.1.A]{Bi} and \cite[Theorem 2.3]{LiU}.
\end{proof}

\begin{prop}\label{prop symp surfaces}
There exists a symplectic form on $M$ such that $F,W,R$ and $S$ are symplectic surfaces.
\end{prop}
From the Gompf sum construction \cite{G} applied to the fibre sum 
\begin{equation*}
E(n)=E(1)\#_{F=F}E(n-1)
\end{equation*}
it is clear that there exists a symplectic form on $M$ such that $F$ and $V$ are symplectic. Hence the surface $W$ is also symplectic. 
\begin{lem}\label{lem R S Lag}
We can choose the surfaces $R$ and $S$ such that they are Lagrangian for a symplectic form from the Gompf construction.
\end{lem}
\begin{proof}
The claim is clear for the rim torus $R$: In the fibre sum construction it is given by $R=\gamma\times \partial D^2$ where $\gamma$ is one of the circle factors of the torus $F=S^1\times S^1$ in a tubular neighbourhood $F\times D^2$ on which the symplectic form is a standard product form. The claim for the vanishing sphere $S$ follows from section 8 in \cite{AMP}.
\end{proof}
Hence Proposition \ref{prop symp surfaces} is a consequence of the following theorem that we formulate in a more general way. The proof is very similar to Lemma 1.6 in \cite{G} due to Gompf which states the same for disjoint Lagrangians.
\begin{thm}
Let $(X,\omega)$ be a closed symplectic 4-manifold and $L_1,\ldots,L_n$ closed connected embedded oriented Lagrangian surfaces in $X$ which intersect each other transversely so that at most two surfaces intersect in any given point of $X$. Suppose that the classes of these surfaces are linearly independent in $H_2(X;\mathbb{R})$. Then there exists a symplectic structure $\omega'$ on $X$, deformation equivalent to $\omega$, such that all of these Lagrangian surfaces become symplectic. We can choose the symplectic structure $\omega'$ such that the induced volume forms on the Lagrangians have any given sign. We can also assume that any symplectic surface disjoint from the Lagrangians stays symplectic.
\end{thm}
\begin{proof}
Let $a_1,\ldots,a_n$ be any real numbers. Since $H^2(X;\mathbb{R})$ is the dual space of second real homology there exists a closed 2-form $\eta$ on $X$ such that
\begin{equation*}
\int_{L_i}\eta=a_i,\quad i=1,\ldots,n.
\end{equation*}
Choose volume forms $\omega_i$ on $L_i$ for each $i$ such that
\begin{equation*}
\int_{L_i}\omega_i=\int_{L_i}\eta.
\end{equation*}
Let $j_i$ denote the embedding of $L_i$ into $X$. There exist 1-forms $\alpha_i$ on $L_i$ such that
\begin{equation*}
\omega_i-j_i^*\eta=d\alpha_i.
\end{equation*}
Let $\pi_i\colon\nu L_i\rightarrow L_i$ denote tubular neighbourhoods and choose cut-off functions $\rho_i(r)$ with support on the tubular neighbourhoods which depend only on the radius $r$ and are $1$ on the zero section. Define 1-forms
\begin{equation*}
\overline{\alpha}_i=\rho_i\pi_i^*\alpha_i
\end{equation*}
on the tubular neighbourhoods. Extend them by zero outside of the neighbourhood and set
\begin{equation*}
\eta'=\eta+\sum_id\overline{\alpha}_i.
\end{equation*}
We claim that $j_i^*\eta'=\omega_i$. This follows if we can show that
\begin{equation*}
j_i^*d\overline{\alpha}_k=0\quad\text{for $k\neq i$}.
\end{equation*}
This is clear if $L_k$ does not intersect $L_i$ by making the tubular neighbourhood of $L_k$ small enough so that it does not intersect $L_i$. Suppose that $L_k$ and $L_i$ intersect in a point $p$. We can assume that $L_i$ intersects $\nu L_k$ in a disk fibre of the tubular neighbourhood. We have
\begin{equation*}
d\overline{\alpha}_k=\rho_k'dr\wedge\pi_k^*\alpha_k+\rho_k\pi_k^*d\alpha_k.
\end{equation*}
By assumption, ${\pi_k}_*$ is the zero map on $T_qL_i$ for each point $q$ on the disk fibre $L_i\cap \nu L_k$. Therefore $d\overline{\alpha}_k$ is zero on any two vectors in $T_qL_i$. Hence $j_i^*d\overline{\alpha}_k=0$.

Consider the closed 2-form 
\begin{equation*}
\omega'=\omega+t\eta'.
\end{equation*}
For small positive $t$ the form $\omega'$ is symplectic. Since the $L_i$ are Lagrangian for $\omega$ we have $j_i^*\omega'=t\omega_i$. Hence the $L_i$ are now symplectic surfaces with (small) positive or negative volume, depending on the sign of $a_i$. 
\end{proof}
\begin{defn}
Let $\omega_0$ denote a symplectic form on $M$ given by Proposition \ref{prop symp surfaces}. We can assume that the symplectic form has the same sign on both $R$ and $S$. Let $T$ denote the symplectic torus of square $0$ obtained by smoothing the intersection between $R$ and $S$. The tori $R$ and $T$ intersect in a single positive transverse point. 
\end{defn}
The surfaces $R$ and $T$ together span a copy of $H$ in the intersection form, which we denote by $H_{RT}$. In summary the intersection form of $M$ is equal to
\begin{equation*}
Q_M=H(n)\oplus H_{RT}\oplus aH\oplus b(-E_8)
\end{equation*}
with certain integers $a,b\geq 1$.
\begin{defn}
We say that a self-diffeomorphism of $M$ satisfies $(\ast)$ if it is the identity on the first summand of $H(n)\oplus H(n)^\perp$. It then preserves the splitting $H(n)\oplus H(n)^\perp$.
\end{defn}

We will frequently use the following proposition that was proved in \cite{H}.
\begin{prop}\label{prop existence diffeom}
Every integral class in $H(n)^\perp$ can be mapped to any integral linear combination of $R$ and $T$ of the same square and divisibility by a self-diffeomorphism of the elliptic surface $M$ that satisfies $(\ast)$. Taking a multiple we see that we can map in this way any rational class in $H(n)^\perp$ to a rational linear combination of $R$ and $T$.
\end{prop}
The following is clear:
\begin{lem}
Let $f\colon M\rightarrow M$ be an orientation preserving diffeomorphism. If $C$ and $D$ are homology classes on $M$ with $f_*C=D$, then $(f^{-1})^*PD(C)=PD(D)$.
\end{lem}
We will now cover a large part of the positive cone by symplectic forms in the following way: We have a symplectic form $\omega_0$ so that the surfaces $F,W,R$ and $T$ are symplectic. The class of $\omega_0$ can be written as
\begin{equation*}
\omega_0=PD(\alpha_0 F+\beta_0 W+\gamma_0 R+\delta_0 T)+Z_0,
\end{equation*}
where $Z_0$ is a class in the real span of $aH\oplus b(-E_8)$. Using inflation with very large parameters and then dividing by a large number it follows that the class
\begin{equation*}
\omega=PD(\alpha F+\beta W+\gamma R+\delta T)
\end{equation*}
plus some arbitrarily small rest is represented by a symplectic form for all positive coefficients $\alpha,\beta,\gamma,\delta$. The second method we use are the actions of self-diffeomorphisms on cohomology. In particular, we can map according to Proposition \ref{prop existence diffeom} any rational class in $H^2(M;\mathbb{R})$ using a self-diffeomorphism to a rational linear combination of the Poincar\'e duals of $F,W,R$ and $T$. This will suffice to prove Theorem \ref{thm PD(F)+} in Section \ref{sect pdf+}, because in this situation all coefficients are positive. To prove the more general Theorem \ref{main thm F} in Section \ref{sect main thm f} we will introduce in Lemma \ref{diffeomorphism fi} another diffeomorphism that allows in some situations to change a negative coefficient in the expansion of $\omega$ into a positive one.

\section{Proof of Theorem \ref{thm PD(F)+} on positive classes}\label{sect pdf+}

We have the following lemma that proves one of the steps outlined above.
\begin{lem}\label{main lem}
Let $\omega$ be a class in $\mathcal{P}^{PD(F)}$. Then there exist a sequence of self-diffeomorphisms $\phi_k$ of the elliptic surface $M$ and classes $\sigma_k$ of the form
\begin{equation*}
\sigma_k=PD(\alpha F+\beta W+\gamma_k R+\delta_k T)
\end{equation*}
with $\beta>0$ such that $\phi_k^*\sigma_k$ converges to the class $\omega$. The diffeomorphisms $\phi_k$ satisfy $(\ast)$. If $\omega$ is a class in the subset $\mathcal{P}^{PD(F)+}$ then we can assume that all coefficients of $\sigma_k$ are positive.
\end{lem}
\begin{proof}
We decompose the class $\omega$ as
\begin{equation*}
\omega=PD(\alpha F +\beta W) +\omega',
\end{equation*}
where $\omega'\in H(n)^\perp$ and $\beta>0$. There exists a sequence $\omega_k'$ of rational classes in $H(n)^\perp$ converging to the class $\omega'$. Using the second part of Proposition \ref{prop existence diffeom} there exist self-diffeomorphisms $\phi_k$ that satisfy $(\ast)$ and map
\begin{equation*}
\phi_k^*PD(\gamma_k R+\delta_k T)=\omega_k'
\end{equation*}
for certain rational numbers $\gamma_k,\delta_k$. Setting 
\begin{equation*}
\sigma_k=PD(\alpha F+\beta W+\gamma_k R+\delta_k T)
\end{equation*}
we get the first claim. If $\omega$ is a class in $\mathcal{P}^{PD(F)+}$ we can assume that all $\omega_k'^2$ are positive. Hence we can assume that $\gamma_k$ and $\delta_k$ are positive.
\end{proof}

Recall that we have a symplectic form $\omega_0$. As above, the class of this form can be written as
\begin{equation*}
\omega_0=PD(\alpha_0 F+\beta_0 W+\gamma_0 R+\delta_0 T)+Z_0,
\end{equation*}
where $Z_0$ is a class in the real span of $aH\oplus b(-E_8)$. We now prove Theorem \ref{thm PD(F)+}.

\begin{proof}
Let $\omega$ be a class in $\mathcal{P}^{PD(F)+}$. Choose a sequence $\sigma_k$ as in Lemma \ref{main lem}. Then
\begin{equation*}
\sigma_k=PD(\alpha F+\beta W+\gamma_k R+\delta_k T)
\end{equation*}
where all coefficients are positive. Consider the symplectic form $\omega_0$ with the symplectic surfaces $F,W,R,T$. We apply the inflation Lemma \ref{inflation} to the form $\omega_0$ which means that we can add to $\omega_0$ any linear combination of the classes $F,W, R, T$ with positive coefficients. This shows that the class
\begin{equation*}
N_k\sigma_k+Z_0
\end{equation*}
is represented by a symplectic form for any sufficiently large positive number $N_k$. Hence also the classes
\begin{equation*}
\eta_k=\sigma_k+\frac{1}{N_k}Z_0
\end{equation*}
are represented by symplectic forms. We know that $\phi_k^*\sigma_k$ converges to $\omega$. We can choose the numbers $N_k$ large enough so that $\frac{1}{N_k}\phi_k^*Z_0$ converges to $0$. Then $\phi_k^*\eta_k$ converges to $\omega$, hence $\omega\in\overline{\mathcal{C}}_M$. 
\end{proof}

\section{Proof of the main Theorem \ref{main thm F}}\label{sect main thm f}

We will use the following lemma which shows that certain automorphisms of the intersection form are realized by self-diffeomorphisms.
\begin{lem}\label{diffeomorphism fi}
For an integer $i$ let $f_i$ denote the map which is the identity on all summands of the intersection form except on $H(n)\oplus H_{RT}$, where it is given by
\begin{align*}
F&\mapsto F\\
W&\mapsto W+iT\\
R&\mapsto R-iF\\
T&\mapsto T.
\end{align*}
Then $f_i$ is induced by a self-diffeomorphism of $M$.
\end{lem}
\begin{proof}
It is easy to check that $f_i$ is an automorphism of the intersection form. The map $f_i$ leaves $F$ and hence $c_1(M)$ invariant. Letting $i$ be a real number and taking $i\rightarrow 0$ we see that $f_i$ has spinor norm one. This implies the claim by the work of Friedman-Morgan \cite{FM}; see also \cite{Lo}.
\end{proof}
We denote a diffeomorphism that induces $f_i$ by the same symbol. The induced automorphism $f_i^*$ on cohomology maps 
\begin{equation*}
\omega=PD((\alpha-i\gamma) F+\beta W+\gamma R+(\delta+i\beta) T)
\end{equation*}
to
\begin{equation*}
f_i^*\omega=PD(\alpha F+ \beta W+\gamma R+\delta T).
\end{equation*}
Note that the class $\omega$ can be positive even if $f_i^*\omega$ is not positive. The main difficulty in the case of Theorem \ref{main thm F} is that we have to approximate classes which are no longer positive by symplectic forms. However, the automorphism $f_i^*$ allows us in some cases to map a positive class to such a non-positive class. The positive class can then be reached by inflation. Hence we have to show that under our assumptions we can always find an integer $i$ such that $f_i^*$ maps a positive class to our given class.

Suppose for example that we want the class $\omega$ as above to be positive. We can assume that $\beta,\gamma>0$. Then $\omega$ is positive if and only if $\alpha-i\gamma>0$ and $\delta+i\beta>0$. This is possible only if 
\begin{equation*}
\alpha\beta+\gamma\delta>0,
\end{equation*}
which is equivalent to $\omega^2>0$ if $M$ is spin and $\omega^2>\beta^2$ if $M$ is non-spin. Note that $\beta=\omega\cdot PD(F)$. This is the reason why we have to restrict to the subset $\mathcal{P}^{PD(F)>}$ in the non-spin case.

We now begin with the proof of Theorem \ref{main thm F}. Fix a cohomology class $\omega$ in $H^2(M;\mathbb{R})$. If $M$ is the elliptic surface $E(2m)$ assume that $\omega$ is in the subset $\mathcal{P}^{PD(F)}$ and if $M$ is the surface $E(2m-1)$ assume that $\omega$ is in the subset $\mathcal{P}^{PD(F)>}$. We want to approximate $\omega$ by symplectic classes. Write
\begin{equation*}
\omega=PD(\alpha F+\beta W)+\omega'
\end{equation*}
where $\omega'$ is an element of the real span of $H(n)^\perp$. The following inequality for the coefficients of the class $\omega$ is a consequence of our assumptions. 
\begin{lem}\label{lem inequ}
We have
\begin{equation*}
\alpha>-\frac{\omega'^2}{2\beta}.
\end{equation*}
\end{lem}
\begin{proof}
In both cases $\beta>0$ and
\begin{align*}
0<\omega^2&=2\alpha\beta+\beta^2W^2+\omega'^2\\
&=2\alpha\beta +\epsilon(n)\beta^2+\omega'^2
\end{align*}
where $\epsilon(n)=0$ if $n$ is even and $\epsilon(n)=1$ if $n$ is odd. If $n=2m$ is even we get
\begin{equation*}
2\alpha\beta>-\omega'^2
\end{equation*}
hence
\begin{equation*}
\alpha>-\frac{\omega'^2}{2\beta}.
\end{equation*}
If $n=2m-1$ is odd we get by the assumption that $\omega$ is in $\mathcal{P}^{PD(F)>}$
\begin{equation*}
\beta^2<\omega^2=2\alpha\beta+\beta^2+\omega'^2.
\end{equation*}
This again implies the claim.
\end{proof}
We now prove a slightly technical lemma. The estimate in (b) will be used in Lemma \ref{lem technical 2} to show that we can find integers $i_k$ such that the automorphisms $f_{i_k}^*$ map a sequence of positive classes to another sequence which can then by mapped by diffeomorphisms to a sequence converging to our given class $\omega$.
\begin{lem}\label{lem technical 1}
There exists a sequence $\omega_k'$ of rational classes in $H(n)^\perp$ converging to $\omega'$ with the following properties:
\begin{enumerate}
\item $\omega_k'^2>\omega'^2$ for all indices $k$.
\item Write $\omega_k'=\frac{1}{A_k}\tau_k$ where $A_k$ is a positive rational number and $\tau_k$ is an indivisible integral class in $H(n)^\perp$. Then there exist integers $i_k$ with
\begin{equation*}
\alpha>\frac{i_k}{A_k}>-\frac{\omega'^2}{2\beta}.
\end{equation*}
\end{enumerate}
\end{lem}
\begin{proof}
Let $\omega_k''$ be any rational sequence in $H(n)^\perp$ converging to $\omega'$. We can assume that 
\begin{equation*}
\omega_k''^2>\omega'^2
\end{equation*}
because every neighbourhood of $\omega'$ contains rational elements with this property. Write
\begin{equation*}
\omega_k''=\frac{1}{B_k}\mu_k
\end{equation*}
where $B_k$ is a positive rational number and $\mu_k$ is integral and indivisible. For each $k$ we can find an integral basis $e_1,e_2,\ldots,e_r$ of the lattice $H(n)^\perp$ such that $e_1=\mu_k$. The basis depends on $k$, but we do not write the index. Let $C_k$ be an arbitrary sequence of positive integers converging to infinity. Consider the rational number $A_k=C_kB_k$ and the integral class $\tau_k=C_k\mu_k+e_2$. Then $\tau_k$ is indivisible. Define
\begin{equation*}
\omega_k'=\frac{1}{A_k}\tau_k=\frac{1}{B_k}\left(\mu_k+\frac{1}{C_k}e_2\right).
\end{equation*}
If we choose the integers $C_k$ large enough the sequence $\omega_k'$ converges to $\omega'$ (note that $e_2$ depends on $k$). Moreover, we can assume that $\omega_k'^2>\omega'^2$. If $C_k$ and hence $A_k$ is large enough we can find by Lemma \ref{lem inequ} an integer $i_k$ such that
\begin{equation*}
\alpha>\frac{i_k}{A_k}>-\frac{\omega'^2}{2\beta}.
\end{equation*}

\end{proof}
Let $\omega_k'=\frac{1}{A_k}\tau_k$ be the sequence from Lemma \ref{lem technical 1}. Since $\tau_k$ is an integral indivisible class in $H(n)^\perp$ we can find by Proposition \ref{prop existence diffeom} a self-diffeomorphism $\phi_k$ of the elliptic surface $M$ satisfying $(\ast)$ such that
\begin{equation*}
\tau_k=\phi_k^*PD(R+\delta_k T)
\end{equation*}
for certain integers $\delta_k$. We get
\begin{equation}\label{eqn omega_k'}
\omega_k'=\phi_k^*PD\left(\frac{1}{A_k}R+\frac{\delta_k}{A_k} T\right).
\end{equation}
This implies that the sequence
\begin{equation*}
\phi_k^*PD\left(\alpha F+\beta W+\frac{1}{A_k}R+\frac{\delta_k}{A_k} T\right)
\end{equation*}
converges to our given class $\omega$. Consider the automorphism $f_i^*$ from Lemma \ref{diffeomorphism fi} and apply $(f_i^{-1})^*$ to the sequence
\begin{equation*}
PD\left(\alpha F+\beta W+\frac{1}{A_k}R+\frac{\delta_k}{A_k} T\right)
\end{equation*}
where $i=i_k$ for the integer $i_k$ from Lemma \ref{lem technical 1}. This implies that there exist self-diffeomorphisms $\psi_k=f_{i_k}\circ\phi_k$ such that $\psi_k^*\sigma_k$ converges to $\omega$, where
\begin{equation*}
\sigma_k=PD\left(\left(\alpha-\frac{i_k}{A_k}\right)F+\beta W +\frac{1}{A_k}R+\left(\frac{\delta_k}{A_k}+i_k\beta\right)T\right)
\end{equation*}
\begin{lem}\label{lem technical 2}
The numbers $\alpha-\frac{i_k}{A_k}$ and $\frac{\delta_k}{A_k}+i_k\beta$ are positive.
\end{lem}
\begin{proof}
The first claim is clear by construction in Lemma \ref{lem technical 1}. Note that by formula \eqref{eqn omega_k'} above
\begin{equation*}
\omega_k'^2=\frac{2}{A_k^2}\delta_k
\end{equation*}
and by construction
\begin{equation*}
\frac{i_k}{A_k}>-\frac{\omega'^2}{2\beta}>-\frac{\omega_k'^2}{2\beta}.
\end{equation*}
Hence
\begin{equation*}
\frac{\delta_k}{A_k}=\frac{1}{2}\omega_k'^2A_k
\end{equation*}
and
\begin{equation*}
i_k\beta>-\frac{1}{2}\omega_k'^2A_k.
\end{equation*}
This implies the second claim.
\end{proof}
Note that all coefficients of $\sigma_k$ are positive. We now argue as in the proof of Theorem \ref{thm PD(F)+}: There exist classes
\begin{equation*}
\eta_k=\sigma_k+\frac{1}{N_k}Z_0
\end{equation*}
represented by symplectic forms such that $\psi_k^*\eta_k$ converges to $\omega$. Hence $\omega\in \overline{\mathcal{C}}_M$. This proves Theorem \ref{main thm F}.

\bibliographystyle{amsplain}

\begin{thebibliography}{999}

\bibitem{AMP} D.~Auroux, V.~Mu\~noz, F.~Presas, {\em Lagrangian submanifolds and Lefschetz pencils}, J.~Symplectic Geom.~{\bf 3}, 171--219 (2005).


\bibitem{Bi} P.~Biran, {\em A stability property of symplectic packing}, Invent.~math.~{\bf 136}, 123--155 (1999).

\bibitem{Bo} J.~Bowden, {\em Symplectic 4-manifolds with fixed point free circle actions}, preprint arXiv:1206.0458.

\bibitem{DL} J.~G.~Dorfmeister, T.-J.~Li, {\em The relative symplectic cone and $T^2$-fibrations}, J.~Symplectic Geom.~{\bf 8}, 1--35 (2010). 

\bibitem{FV1} S.~Friedl, S.~Vidussi, {\em Twisted Alexander polynomials detect fibered 3-manifolds}, Ann.~of Math.~{\bf 173}, 1587--1643 (2011).

\bibitem{FV2} S.~Friedl, S.~Vidussi, {\em A vanishing theorem for twisted Alexander polynomials with applications to symplectic 4-manifolds}, J.~Eur.~Math.~Soc.~(JEMS) {\bf 15}, 2127--2041 (2013). 

\bibitem{FM} R.~Friedman, J.~W.~Morgan, {\sl Smooth four-manifolds and complex surfaces}, Ergebnisse der Mathematik und ihrer Grenzgebiete (3) 27, Springer-Verlag, Berlin, 1994.

\bibitem{Ge} H.~Geiges, {\em Symplectic structures on $T^2$-bundles over $T^2$}, Duke Math.~J.~{\bf 67}, 539--555 (1992).

\bibitem{G} R.~E.~Gompf, {\em A new construction of symplectic manifolds}, Ann.~of Math.~{\bf 142}, 527--595 (1995).

\bibitem{GS} R.~E.~Gompf, A.~I.~Stipsicz, {\sl $4$-manifolds and Kirby calculus}, Graduate Studies in Mathematics, 20. Providence, Rhode Island. American Mathematical Society 1999.

\bibitem{H} M.~J.~D.~Hamilton, {\em The minimal genus problem for elliptic surfaces}, preprint arXiv:1206.1260v1.

\bibitem{LaMc} F.~Lalonde, D.~McDuff, {\em The classification of ruled symplectic $4$-manifolds}, Math.~Res.~Lett.~{\bf 3}, no.~6, 769--778 (1996).

\bibitem{Li} T.-J.~Li, {\em The space of symplectic structures on closed 4-manifolds}, AMS/IP Studies in Advanced Mathematics 42, 259--273 (2008).

\bibitem{LiLiu} T.-J.~Li, A.-K.~Liu, {\em Uniqueness of symplectic canonical class, surface cone and symplectic cone of 4-manifolds with $b^+= 1$}, J.~Differential Geom.~{\bf 58}, 331--370 (2001).


\bibitem{LiU} T.-J.~Li, M.~Usher, {\em Symplectic forms and surfaces of negative square}, J.~Symplectic Geom.~{\bf 4}, 71--91 (2006).

\bibitem{Lo} M.~L\"onne, {\em On the diffeomorphism groups of elliptic surfaces}, Math.~Ann.~{\bf 310}, 103--117 (1998). 

\bibitem{McD} D.~McDuff, {\em Notes on ruled symplectic 4-manifolds}, Trans.~Amer.~Math.~Soc.~{\bf 345}, 623--639 (1994).

\bibitem{T} C.~H.~Taubes, {\em The Seiberg-Witten and Gromov invariants}, Math.~Res.~Lett.~{\bf 2}, no.~2, 221--238 (1995).

\end{thebibliography}

\bigskip
\bigskip

\end{document}